\documentclass[options]{amsart}
\usepackage{latexsym,ifthen,amssymb}
\usepackage{amsmath}
\usepackage{hyperref}
\usepackage[toc,page,title,titletoc,header]{appendix}
\usepackage{enumerate}
\usepackage{graphicx}
\usepackage{bm}
\usepackage{subfigure}
\usepackage{float}

\usepackage{tikz-cd}
\usepackage{caption}

\usepackage{multirow}
\usepackage{mathrsfs}

\usepackage{hyperref}

\usepackage{cleveref}
\newtheorem{theorem}{Theorem}[section]
\newtheorem{lemma}[theorem]{Lemma}

\newtheorem{claim}[theorem]{Claim}
\newtheorem{corollary}[theorem]{Corollary}

\theoremstyle{definition}
\newtheorem{definition}[theorem]{Definition}

\theoremstyle{remark}

\numberwithin{equation}{section}

\theoremstyle{noparens}
\newtheorem*{question*}{Question}
\newtheorem*{theorem*}{Theorem}

\def\mcal{\mathcal}
\def\h{\hat}
\def\t{\tilde}

\def\C{C}

\def\m{m}
\def\capacity{capacity}
\def\uhr{\upharpoonright}

\def\alice{Alice}
\def\baby{Baby}

 \title{Martin-L\"of randomness and rK-reducibility}

 \author{George Barmpalias    }
 \address{State Key Lab of Computer Science, Institute of Software, Chinese Academy of Sciences, Beijing, China. }
\email{barmpalias@gmail.com }

  \author{Lu Liu }
\address{School of Mathematics and Statistics,
Central South University,
City Changsha, Hunan Province,
China. 410083}
\thanks{Lu Liu is supported by  NSFC grant 12471002.}
\email{g.jiayi.liu@gmail.com}

\begin{document}

\maketitle
\begin{abstract}
Can random real carry 
arbitrary information. In terms 
of (Turing) computability, Ku\v{c}era and G\'{a}cs
addressed this question in affirmative. i.e., 
Every real $x\in 2^\omega$ is computable from some Martin-L\"of random real.
This result  
motivated quantitative versions of coding into random reals.
i.e., How many bits of the random oracle are needed to
compute the first $n$ bits of $x$.
Barmpalias and   Lewis-Pye obtained a coding that uses
$n+O(\log n)$ many bits.
Is it possible for the coding to use $n$ bits exactly?
This is equivalent to say, if every real is $rK$-reducible to 
a Martin-L\"of random real.
In this paper, we prove that such a coding does not exist.
i.e., There exists some real that is not $rK$-reducible to 
any Martin-L\"of random real.

\end{abstract}

 %\section{Introduction}

\section{Introduction}

One of the basic questions in algorithmic randomness theory is how much
information a random real can contain. There are two different aspects
to this question. We may ask what can be computed from a random real.
We may also ask how much of the random real must be known in order to
describe the information that it contains. The distinction matters:
an oracle computation may read much farther into the oracle than the
part of its output that we want to produce.

The first question has a remarkable answer. Ku\v{c}era and G\'{a}cs
proved that every real is computable from some Martin-L\"of random real
\cite{Kucera1985,Gacs1986}. Thus randomness does not prevent a real
from carrying arbitrary noncomputable information. This result also
motivated quantitative versions of coding into random reals. In these
versions one asks how many bits of the random oracle are needed to
compute a given number of bits of the real being coded. The amount of
extra oracle use, usually called redundancy, has been studied from
several directions. Doty related asymptotic compression into random
oracles to effective dimension \cite{Doty2006}. Later work of Barmpalias and Lewis-Pye obtained
optimal logarithmic redundancy: in particular, there are coding methods
that use only $n+O(\log n)$ bits of a random oracle to compute the first
$n$ bits of an arbitrary real
\cite{BarmpaliasLewisPye2018Optimal,BarmpaliasLewisPye2018}.
None of these questions, however, asks
whether an initial segment of the random oracle already gives a
bounded-length description of the corresponding initial segment of
the coded real. This is the question considered here.

There is another route to this question, starting from relative
randomness rather than oracle computation. Martin-L\"of randomness
was defined by the avoidance of effective tests of small measure
\cite{MartinLof1966}.
It can also be recognized from the prefix-free Kolmogorov complexity
of initial segments: a real is 1-random precisely when its initial
segments cannot be compressed by more than a fixed number of bits
\cite{Schnorr1971,Chaitin1975}.
The complexity viewpoint allows us to compare two reals even when
neither of them is computable or random. We can ask whether the
initial segments of one real are uniformly no more complex than those
of another. This gives the ordering usually called $K$-reducibility
\cite{DowneyHirschfeldtLaForte2004,MillerYu2008}. It compares two
complexity profiles, but it does not say that one initial segment can
be described \emph{from} the other.

In the class of left-c.e. reals, there is a related and older way to
compare randomness. Solovay reducibility compares the rates at which
computable increasing approximations converge. The halting
probabilities of universal prefix-free machines are left-c.e. and
1-random, and their Solovay degrees are greatest among the left-c.e.
reals \cite{Chaitin1975,CaludeEtAl2001,KuceraSlaman2001}.
The Solovay degrees of left-c.e. reals are also dense below their
greatest degree \cite{DowneyHirschfeldtNies2002}.
Solovay reducibility preserves inequalities between initial-segment
complexities. On the other hand, its definition is
tied to increasing approximations. To compare arbitrary reals by the
information in their initial segments, Downey, Hirschfeldt and
LaForte introduced $rK$-reducibility
\cite{DowneyHirschfeldtLaForte2004}. It uses conditional complexity:
instead of merely comparing the lengths of shortest descriptions,
it asks for a short description of one initial segment when the
corresponding initial segment of the other real is given.

We now state these comparisons precisely. We regard elements of
$2^\omega$ as reals, and write $x\uhr n$ for the first $n$ bits of a
real $x$. Let $K(\sigma)$ be the prefix-free complexity of a finite
string $\sigma$, and let $K(\sigma\mid\tau)$ be its conditional
prefix-free complexity given $\tau$. The Levin--Schnorr
characterization of 1-randomness says that
\[
  y\text{ is 1-random}
  \quad\Longleftrightarrow\quad
  K(y\uhr n)\geq n-O(1)\text{ for every }n.
\]
For reals $x,y$, we write $x\leq_K y$ if
\[
  K(x\uhr n)\leq K(y\uhr n)+O(1)
  \quad\text{for every }n.
\]
We say that $x$ is \emph{$rK$-reducible} to $y$, and write
$x\leq_{rK}y$, if
\[
  K\bigg(x\uhr n~\bigg|~ y\uhr n\bigg)\leq O(1)
  \quad\text{for every }n.
\]
In each case the constant may depend on $x$ and $y$, but not on $n$.
This uniformity across all lengths is essential.

The original definition of $rK$-reducibility has a useful coding
interpretation. Given $y\uhr n$, an effective procedure is allowed
to present a bounded list of candidates for $x\uhr n$; the correct
candidate must be on the list, but it need not be singled out
effectively. The bound on the size of the list is independent of $n$.
Equivalently, one of a bounded collection of short descriptions
produces $x\uhr n$ from $y\uhr n$. The choice of a description may
change with $n$. Hence $rK$-reducibility is less uniform than a
single procedure that always computes $x\uhr n$ from $y\uhr n$,
but it still makes a strong demand on the two initial segments
\cite{DowneyHirschfeldtLaForte2004}.

There are two useful comparisons. Downey, Hirschfeldt and LaForte
proved that $x\leq_{rK}y$ implies $x\leq_T y$. Also, a bounded
conditional description gives
\[
  x\leq_{rK}y\quad\Longrightarrow\quad x\leq_K y.
\]
The converse of the second implication is false: knowing only that
the initial segments of $y$ have at least as much complexity as
those of $x$ does not supply short descriptions of $x\uhr n$ from
$y\uhr n$. For example, there are noncomputable $K$-trivial reals
$x$, for which $K(x\uhr n)\leq K(n)+O(1)$ for every $n$.
If $y$ is computable, then $K(y\uhr n)=K(n)+O(1)$, so $x\leq_K y$.
But $x\not\leq_{rK}y$, since $rK$-reducibility implies Turing
reducibility \cite{DowneyHirschfeldtLaForte2004}.
The $K$-trivial reals are also characterized by lowness for
Martin-L\"of randomness \cite{Nies2005}, which shows that this example
lies within another central theme of algorithmic randomness.
These comparisons
place $rK$-reducibility between two different ways of measuring
information: computation from an oracle and comparison of
complexity profiles.

The structure of the $K$-degrees has its own difficulties. Miller
and Yu showed that the two halves of a 1-random join need not have
an upper bound in the $K$-degrees \cite{MillerYu2008}. They also
studied the oscillations of the initial-segment complexity of random
reals and obtained further results about random $K$-degrees
\cite{MillerYu2011}. For 2-random reals, the upper cones in the
$K$-degrees are countable \cite{Miller2009}. Recent work of Liu \cite{Liu2026} shows that
there is no maximal $K$-degree. These are results
about the numerical complexity ordering. Our question concerns the
stronger relation obtained by putting one initial segment in the
condition of a description of the other.

The results for left-c.e. reals might suggest that random reals
are rich enough to sit above all reals in the $rK$ ordering. Indeed,
Downey, Hirschfeldt and LaForte proved that a left-c.e. real is
$rK$-complete among left-c.e. reals exactly when it is 1-random
\cite{DowneyHirschfeldtLaForte2004}. Together with the
Ku\v{c}era--G\'{a}cs theorem, this raises a natural question: can
\emph{every} real be placed below \emph{some} 1-random real in the
$rK$ ordering? The quantifiers are important. We may choose the
random real separately for each real, just as in the
Ku\v{c}era--G\'{a}cs theorem. The question is whether this freedom
is enough to obtain bounded conditional descriptions at every length.

Our result gives a negative answer.

% For reals $x,y\in 2^\omega$, we say 
% $x$ is \emph{$rK$-reducible} to $y$ iff 
% $$K\bigg(x\uhr m~\bigg|~y\uhr m\bigg)  \leq O(1).$$
%Where the $O(1)$ constant depends on $x,y $ but not $m$.

\begin{corollary}\label{optth0}
There is a  real $x\in 2^\omega$
such that: for every 1-random real $y\in 2^\omega$,
we have $x$ is not rK-reducible to $y$.
That is:
$$\sup_m K\bigg(x\uhr m~\bigg|~y\uhr m\bigg) = \infty.$$
\end{corollary}
Obviously, the witness $x$ is automatically
non 1-random.
For this $x$, the Ku\v{c}era--G\'{a}cs theorem still provides a
1-random $y$ that computes $x$. The theorem says that no such $y$,
and in fact no 1-random $y$ at all, gives uniformly bounded
conditional descriptions at matching lengths.
Nor can $x$ be left-c.e., by the completeness result above. Thus
the obstruction is not visible inside either of those two classes.
The statement concerns $rK$-reducibility specifically; it does not
assert the corresponding failure for the weaker $K$-reducibility.

The rest of this paper will prove Corollary \ref{optth0}.
We firstly reduce Corollary \ref{optth0} to a finite game question
(Section \ref{optseccodinggame}).
Then we derive the winning strategy for this game in the rest sections.
We end up this section by introducing some general notations.

\noindent \textbf{Notation.}
For $\sigma\in 2^{<\omega}\cup 2^\omega$, 
%we write $\sigma\uhr_m^n$ for the string $\sigma(m)\sigma(m+1)\cdots\sigma(n)$;
we write $\sigma\uhr n$ for the string $\sigma(0)\cdots\sigma(n-1)$;
we write $\sigma\succeq\rho$ if $\sigma$ extends $\rho$;
let $[\sigma]^\preceq $ be the set of strings extending $\sigma$;
%let $[\sigma,\h\sigma]$ denote the set of strings $\sigma\preceq\t\sigma \preceq\h\sigma$;
for a set $A\subseteq 2^{<\omega}$, let $[A]^\preceq = \cup_{\sigma\in A}[\sigma]^\preceq$;
write $[\sigma] = \{x\in 2^\omega:x\succeq\sigma\}$
and   $[A] = \cup_{\sigma\in A}[\sigma]$.
We often treat a natural number $n$ as the set $\{0,\cdots, n-1\}$.

For a set $A\subseteq 2^{<\omega}$, the measure of $A$,
denoted as $\m(A)$, refers to
the Lebesgue measure of $[A]$;
for $\sigma\in 2^{<\omega}$, we write $\m(\sigma)$ for $\m(\{\sigma\})$;
we write $\m(A|B)$ for $\m(A\cap B)/\m(B)$.

 %For a function $g$ on $2^{<\omega}$, we write $\int_{\sigma\in A}g(\sigma)$
%(or simply $ \int_A g$) for $\sum_{\sigma\in A}g(\sigma)\m(\sigma)$.
%Sometimes we use $A$  to denote the cardinality of $A$, so it makes sense
%to write $A\pm B$ etc (we use $A\setminus B$ to denote set minus).

% A  function $f:2^{\leq n}\rightarrow \fin(2^{\leq n})$ is \emph{\treelike}
% iff for every $\rho\in 2^{\leq n}$,
% $f(\rho) = \cup_{\h\rho\in [\rho]^\preceq\cap 2^n} f(\h\rho)\uhr |\rho|$.

\section{A coding game}
\label{optseccodinggame}

In this section, we will introduce a finite game and reduce 
Corollary \ref{optth0} to the winning strategy 
of that game.
This is natural in computability theory.
In this game, one player \alice\ controls a Martin-Lof 
test and the other player \baby\ controls the Kolmogorov
complexity function $K[t]$.
To simplify things, we don't need to record the full details of
the  complexity function
presented by \baby.
Note that given $a\in \omega$, 
for very $\rho\in 2^{<\omega}$,
there are at most $2^a$ many $\sigma\in 2^{|\rho|}$
such that $$K\bigg(\sigma\uhr m~\bigg|~\rho\uhr m\bigg)\leq a
\text{ for all }m\leq |\rho|.$$
For each $\rho$, \alice\ only concern the set of such $\sigma$.
We refer to $\rho,\sigma$ as \emph{code} and \emph{source} respectively.
We start with some notations of this game
and then give the concrete definition of this game in Definition \ref{optgame1}.

\noindent\textbf{Notations of forest and coding function.}
\begin{itemize}
\item For $p,n\in\omega$,
we think of set $p\times 2^{\leq n}$
as a forest (namely $p$ many trees).
For $(\h p,\sigma)\in p\times 2^{\leq n}$,
we usually don't need to specify $\h p$,
thus we often write 
 $\sigma$ instead of $(\h p,\sigma)$;
write $\sigma\preceq\tau$ instead of $(\h p,\sigma)\preceq(\h p,\tau)$;
and so forth. 
We let $\m$ denote the  measure on $p\times 2^{ n}$
so that $\m(\sigma) = 2^{-n}$ for all $\sigma\in p\times 2^n$.

 \item A \emph{coding function} from $q\times 2^{\leq n}$ to $p\times 2^{\leq n}$
  is a function
$$f:q\times 2^{\leq n}\rightarrow \mcal{P}(p\times 2^{\leq n})$$
such that for every $\rho\in q\times 2^{\leq n}$,
$f(\rho)\subseteq 2^{|\rho|}$ and
if $\sigma\in f(\rho)$   then prefix of 
$\sigma$ is also coded by the same-length prefix of $\rho$.
That is:
$$\sigma\uhr m\in f(\rho\uhr m)\text{ for all $m\leq |\sigma|$}.$$
We say $\sigma$ is \emph{$f$-coded}
(or simply say coded when there is no ambiguous) by $\rho$
iff $\sigma\in f(\rho)$.
A coding function $f$
is $a$-\emph{\capacity}\ iff for every $\rho\in dom(f)$,
$|f(\rho)|\leq a$.
For coding functions $\h f, f$ on $q\times 2^{\leq n}$,
we say $\h f$ \emph{extends} $f$ iff
for every $\rho\in q\times 2^{\leq n}$, $\h f(\rho)\supseteq f(\rho).$ 

\end{itemize}

Now comes the coding game. Let $p,q,a,n\in\omega^{>0}$.
\begin{definition}[$(p,q,a,n)$-coding-game]
\label{optgame1}
 At each round $t\in\omega$, \alice\ presents a set
 $ A[t]\subseteq q\times 2^n$
 with $A[t]\supseteq A[t-1]$ (when $t>0$); then \baby\ presents
a $a$-\capacity\ coding function $f[t]$ from $q\times 2^{\leq n}$
to $p\times 2^{\leq n}$ 
extending $f[t-1]$ (when $t>0$) such that,
\begin{itemize}
\item every source string $\sigma\in p\times 2^n$ admits
 a code $\rho\in q\times 2^n\setminus A[t]$. i.e.,
$\sigma \in f[t](\rho )$.
\end{itemize}
\alice\ \emph{wins} if at some round $t$, \baby\ violates the game rule.

\end{definition}
Let $A$ be the set \alice\ presents at the round she wins.
We show that \alice\ has a winning strategy with an arbitrary small cost.

\begin{theorem}\label{optpropwin}
For every $a,p,q\in\omega^{>0}$, every $c>0$, \alice\ has a winning strategy for the $(p,q,a,n)$-game \ref{optgame1} (for some $n$)
such that $\m(A)\leq c$.
\end{theorem}
It is clear that if \alice\ has a winning 
strategy for the $(p,q,a,n)$-game \ref{optgame1}
with $\m(A)\leq c$, then the same holds with $n$ replaced by 
any $N>n$ since \alice\ could apply the winning strategy on level $n$.
Although winning $(1,q,a,n)$-game \ref{optgame1} implies that
 of $(p,q,a,n)$-game \ref{optgame1}, the parameters here are necessary for 
 proving Theorem \ref{optpropwin}.
We end up this section by
introducing more notation
and reducing Corollary \ref{optth0} to Theorem \ref{optpropwin}.

\noindent\textbf{More notation and terminology.}
Usually, we use $\rho $ to denote a code string and $\sigma$ to denote a
source string.
In a $(p,q,a,n)$-game \ref{optgame1}:
\begin{itemize}

 \item For a string $\sigma\in p\times 2^{<\omega}$, we write 
 $[\sigma]^{\leq n}$ for the set $(p\times 2^{\leq n})\cap [\sigma]^\preceq$
 (when there is no ambiguity);
write  $[\sigma]^{=n}$ for the set $(p\times 2^n)\cap [\sigma]^\preceq$;
 write $[\sigma]^{[m,n]}$ for the set $(p\times 2^{[m,n]})\cap [\sigma]^\preceq$.
 For a set $C\subseteq p\times 2^{<\omega}$, we write $[C]^{\leq n}$
 for the set $\cup_{\sigma\in C}[\rho]^{\leq n}$;
 similarly for $[C]^{=n}$ etc.
 
\item We say   $\sigma$ is   \emph{coded}  by $\rho$ (at round $t$) iff: $\sigma \in f[t](\rho )$;
 $\sigma$ is \emph{coded over $\t\rho$} (at round $t$) iff
for some $\rho\succeq \t\rho$, $\sigma\in f[t](\rho)$.
\item We let $f^{-1}(\sigma)$ denote the set of $f$-codes of $\sigma$.
For a set $C$ of codes,
we write $f(C)$ for the set $\bigcup_{\rho\in C}f(\rho)$.
A code $\rho$ is \emph{nonempty} (at round $t$) iff $f[t](\rho)\ne\emptyset$.
%$\sigma$ is \emph{validly coded by }$\rho$ (at round $t$) if $\rho\notin A[t]$ and $\sigma\in f[t](\rho)$.
 
 \item A code $\rho\in q\times 2^n$ is \emph{valid} (at round $t$) if $\rho\notin A[t]$
 (a code $\rho\in q\times 2^{\leq n}$ is \emph{valid} iff $[\rho]\setminus  [A[t]]\ne\emptyset$).
 We say $\sigma$ is \emph{validly coded over} $\rho$
 (or $\sigma$ is a \emph{valid source over} $\rho$)
(at round $t$) if there is some source string
in $  [\sigma]^{=n}$ coded by some valid code over $\rho$.
i.e., The code is in $   [\rho]^{=n}\setminus A[t].$

\item 
For real variable $u$, we write
$o_{u\rightarrow\infty}( 1)$ for some 
variable $v$ that $v\rightarrow 0$ as $u\rightarrow\infty$.
We write $o_{w,u\rightarrow\infty}(1)$
for some variable $v_w$ depending on $w$, we have for every
$w$,
 $v_w\rightarrow 0$ as $u\rightarrow\infty$.
We say:
for every $v$  \emph{sufficiently large depending on $u$},
$\psi(v,u)$ holds
iff there is a function $h$ such that
for every $u$ and every $v\geq h(u)$,
$\psi(v,u)$ holds.

\end{itemize}

\begin{proof}[Proof of Corollary \ref{optth0} assuming Theorem \ref{optpropwin}]
The connection between Game \ref{optgame1} and Corollary \ref{optth0} is the as follows.
Note that for every $a\in \omega$,
every $\rho\in 2^{<\omega}$, there are at most $2^a$ many 
$\sigma\in 2^{|\rho|}$ such that 
$$K\bigg(\sigma\uhr m~\bigg|~\rho\uhr m\bigg)\leq a
\text{ for all $m\leq |\rho|$}.$$
Thus, 
let $f_a[t](\rho)$ be the set of $\sigma\in 2^{|\rho|}$,
such that $$K[t]\bigg(\sigma\uhr m~\bigg|~\rho\uhr m\bigg)\leq a
\text{ for all $m\leq |\rho|$},$$
we have $f_a[t]$ is a computable sequence of coding function with 
capacity $2^a$ such that $f_a[t]$ extends $f_a[t-1]$.

\begin{itemize}
\item
For $a=1$, we can pick a sufficiently large level (to be specified)  $n_1\in\omega$
and applies the winning strategy for 
$(1,1,2^a,n_1)$-game \ref{optgame1} (where we take $f_a[t]$ as \baby's 
action)
with a small cost $\m(A_1)\leq 2^{-1}$
(where $A_1\subseteq 2^{n_1}$ is the set of strings we enumerate).
In the end, there will be 
a $\sigma^*\in 2^{n_1}$ (referred as \emph{winning string
at hierarchy 1}) that no longer admits any 
valid $f_a$-code. This entails that $$K\bigg(\sigma^*\uhr m~\bigg|~\rho\uhr m\bigg)>1$$
for some $m\leq n_1$ for all $\rho\notin A_1$.
Refer to this as the first hierarchy game.

\item  For $a=2$, we pick a sufficiently larger 
(to be specified) level $n_1+n_2$
and for each $\sigma\in 2^{n_1}$, consider source tree  $   [\sigma]^{\leq n_1+n_2}$  
versus code space
(consisting of $2^{n_1}$ many code trees) $(  [\rho]^{\leq n_1+n_2}: \rho\in 2^{n_1})$.
Each source tree give rise to a $(1,2^{n_1}, 2^a,n_2)$-game \ref{optgame1}
so there are $2^{n_1}$ many games. 
For each of these games,
 apply winning strategy for 
$(1, 2^{n_1}, 2^a, n_2)$-game \ref{optgame1}
so that the overall cost
of all these games satisfy $\m(A_2)\leq 2^{-2}$
(this is made possible since $n_2$ is sufficiently large depending on $n_1$).
In the end, 
for every $\sigma\in 2^{n_1}$,
there will be a $\sigma^*\in  [\sigma ]^{=n_1+n_2}$
(referred as winning string of $\sigma$ at hierarchy $2$)
that no longer admits any valid $f_a$-code.
This entails that $$K\bigg(\sigma^*\uhr m~\bigg|~\rho\uhr m\bigg)>2$$
for some $m\leq n_1+n_2$ for all $\rho\notin A_2$.
Refer to this as the second hierarchy game.

\item 
Similarly for other $a\in\omega$ addressed by games
at higher hierarchy.
\end{itemize}

 We emphasise that the game played in different hierarchy are 
 independent of each other.
 Let $\sigma_1^*\in 2^{n_1}$ be a winning string at hierarchy 1
 and $\sigma_2^*\in  [\sigma_1^*]^{=n_1+n_2}$ be a winning string 
 of $\sigma_1^*$ at hierarchy $2$ and so forth;
 and let $$x:= \cup_{a\in\omega^{>0}} \sigma^*_a.$$
 Let $y$ be 1-random. 
 It suffices to show that $$\sup_m K\bigg(x\uhr m~\bigg|~y\uhr m\bigg)=\infty.$$
 Since $\m(A_a)\leq 2^{-a}$ and $(A_a:a\in\omega^{>0})$ is c.e. uniformly in $a$,
there  exists some $a^*$ such that $y\notin [A_{  a}]$ for all $  a\geq a^*$.
 By our construction of $x$ (recall the definition of winning string)
 for every $a\geq a^*$, since $y\notin [A_a]$, there exists
 $m$ such that $$K\bigg(x\uhr m~\bigg|~y\uhr m\bigg)>a.$$
 Thus, we are done.

\end{proof}

The rest of this paper will prove Theorem \ref{optpropwin}.

\noindent\textbf{Organization.}
The proof proceeds by induction on $a$.
We firstly prove the theorem for capacity one in Section \ref{optseccapacityone}.
Then we prove
the result for  capacity two  in Section \ref{optseccapacitytwo}.
This serves as a easier example for how the induction proceeds 
from $a-1$ to $a$.
In Section \ref{optsecgeneralcapacity}, we prove Theorem \ref{optpropwin} for 
arbitrary capacity.

\section{Capacity one}
\label{optseccapacityone}

Our goal in this section is a winning strategy for capacity one
(for arbitrary $p,q$) with arbitrary small cost:

\begin{lemma}
\label{optlemcapacityone}
For every $p, q\in\omega^{>0}$, every $c>0$,
\alice\ has a static\footnote{A static strategy is a strategy that enumerate
everything in one round. This strategy unfortunately does not give
a winning strategy for capacity more than one.
} winning strategy 
for 
$(p,q,1,n)$-coding-game (for some $n$) such that $\m(A)\leq c$.
\end{lemma}
 \begin{proof}
 Without loss of generality, set $p=1$.
 Let $m$ be sufficiently large (depending on $q$)
 and $\t k$ be sufficiently large depending on $m$ (to be specified).
 Let $I_0,I_1,\cdots, I_{\t k-1}\subseteq n$ be mutually disjoint integer intervals
 with $|I_k| = m$ for all $k<\t k$.
 For a string $\rho$ and set $I\subseteq\omega$, we write $\rho_I$ to denote 
 the   restriction of $\rho $
 (as a function $|\rho|\rightarrow 2$) on $I$.
 Consider the 
 set $$
  A:= \{\rho\in q\times 2^n: \rho_{I_{ k}}\ne 0^{I_{ k}}\text{ for all } k< \t k\}.
 $$
We show that 
once \alice\ enumerates $A$, \alice\ wins.
For notation simplicity, we assume $I_0 = [0,m-1 ]$. 

We prove by induction on $ \t k$.
For $\t k = 1$. Note that (upon enumerating $A$) there are 
$q$ many valid codes at level $m$, so these codes can code 
at most $q$ many source strings; meanwhile
there are $2^m$ many source strings.
Thus, since $2^m>q$, we have \alice\ wins when $\t k = 1$.
Suppose \alice\ wins by enumerating $A$ when $\t k\leq \t k-1$
and now $\t k = \t k$.
Again,  at level $m$,
there are at least $ 2^m-q$ many source strings
coded by ($q(2^m-1)$ many) prefix of $A$
(and not coded by any non-prefix-of-$A$ codes).
 Since $m$ is sufficiently large,
 there is a source 
 that is coded by at most $q$ many prefix-of-$A$ codes 
 (and is not coded by any non-prefix-of-$A$ codes).
 It follows by induction hypothesis
(that \alice\ wins when $\t k =  \t k-1$)
that \alice\ wins.

 On the other hand, the cost of \alice\ is 
 $$
 \m(A) =  q(1-1/2^m)^{\t k}\leq c.
 $$
 The last inequality is due to $\t k$ sufficiently large
 depending on $m$.
 
 \end{proof}

 \section{Capacity two}
 \label{optseccapacitytwo}
 
 Our goal in this section is a winning strategy for capacity two
(for arbitrary $p,q$) with arbitrary small cost:
 
\begin{lemma}\label{optlemcapacity2multicodetree}
For 
every $p,q\in\omega^{>0}$, every $c>0$, \alice\ has a winning strategy for the
$(p,q,2,n)$-game \ref{optgame1}  (for some $n$)
such that $\m(A)\leq c$.
\end{lemma}

 \noindent\textbf{Induction  framework.}
 The proof proceeds by induction on $q/p$.
 For $q/p\leq 1/2$,
 the proof is obvious.
 There  are $\geq 2q2^n$ many sources (in level $n$)
 and $q2^n$ many codes while each code can code at most $2$ sources.
 Thus, enumerating any code wins the game.
 Suppose   Lemma \ref{optlemcapacity2multicodetree}  holds for $q/p\leq (r-1)/2$
 and now $q/p\leq r/2$ (where $2\leq r\in\omega$).
 The winning strategy consists of two phases. In the first phase,
 \alice\ forces that every source (at some middle level $\t n$
 depending on \baby's response in this phase) is coded in
 a smaller coding space (where this smaller coding space satisfies $q/p\leq (r-1)/2$);
 in the second phase, she applies the winning strategy for the case $q/p\leq (r-1)/2$
 on this smaller coding space. The first phase is the key. We explain the intuition for 
 the first phase strategy. 

\noindent \textbf{Intuition for the first phase.}
  \alice\ will firstly forces each source in level $N$
 to have more than $(r-1)/2$ many codes
 (the cost is arbitrary small by induction hypothesis on
 the case $p/q\leq (r-1)/2$).
 Then
 (without \alice\ doing anything)
  there will be  a non-small set $\C_1$ in some middle level $m$
 and another level $m+\h n$ so that over $\C_1$ up to level $m+\h n$,
 almost all codes are non empty.
 \alice\ therefore, enumerates a small set to entail that every valid
 code over $\C_1$ below level $m+\h n$ is non empty. The key note here is:
 \begin{align}\nonumber
 &\text{If }\sigma\in p\times 2^{|\rho|}\text{ is not currently coded by }\rho\text{ and 
 every valid}\\ \nonumber
 &\text{code over $\rho$ is nonempty, then the coding space 
 over $\rho$ has}\\ \nonumber
 &\text{valid capacity \emph{one} for the source
 space $[\sigma]^\preceq$.}
 \end{align}
 This enables \alice\ to invoke the winning strategy for capacity one
game to entail that enough many source is coded over $\C_0$
(the complement of $\C_1$ at level $m$).  
On the other hand, $\C_0$ is promised to be not too large and therefore
enables \alice\ to apply the winning strategy for the case $p/q\leq (r-1)/2$.

 \def\C{C}
 \ \\
 
\noindent \textbf{Wining strategy sketch.}
\begin{enumerate}
\item 
First phase:
\alice\ will apply a strategy (with considerable effort and to be specified)
to achieve the following.
For some    $m\in\omega$ (to be specified)  
and sufficiently large  (depending on $m$ and to be specified) $\t n\in\omega$,
the following holds.
There is a partition of the coding space $q\times 2^m$ into two parts,
 namely $\C_0$ and $ \C_1$
 such that $|\C_1|\geq \frac{q}{2}\cdot 2^m$
 (so $|\C_0|\leq \frac{q}{2}\cdot 2^m\leq\frac{r-1}{2}p\cdot 2^m$).
Furthermore,\footnote{We mean in all future round after the strategy is applied.
Moreover, \alice\ need not go back to coding space $\C_1$
(to maintain (\ref{optphaseonetarget})) once the strategy is applied.} 
\begin{align}\label{optphaseonetarget}
&\text{No source $\sigma\in p\times 2^{m+\t n}$ 
can be solely }\\ \nonumber
&\text{validly coded in $[\C_1]^\preceq$ in the future.
}
\end{align}

\item Second phase:
Now \alice\ treat each $ \rho\in \C_0$
as a root of a coding tree, 
each $\sigma\in p\times 2^m$ 
as the root of a source tree, and apply the winning strategy 
(for $(p2^m, [\frac{(r-1)p}{2}2^m], 2, \t n)$-game \ref{optgame1})
at level $m+\t n$.
By (\ref{optphaseonetarget}), the induced game
complies to the definition of Game \ref{optgame1}
(in the sense: every source in level $m+\t n $ is validly coded over $\C_0$).
By induction hypothesis (for the case $p/q\leq (r-1)/2$),
\alice\ wins the game (with an arbitrary small cost
since $\t n$ is sufficiently large depending on $m$).

\end{enumerate}

Clearly, it remains to address the strategy of the first phase
which is the key of the proof.

 \ \\

\noindent\textbf{First phase strategy.}

\begin{enumerate}
\item First phase step 1:
force at least $  r/2 $ many codes for each source in level $N<n$. 
For each source $\sigma\in 2^N$,  before $\sigma$ admits
more than $ (r-1)/2$ many codes,
apply the winning strategy (promised by inductive hypothesis)
with respect to source tree $ [\sigma]^{\leq n}$
and (at most) $(r-1)/2$ many code trees (whose root are
codes of $\sigma$) to force $\sigma$ to admit
more than $ (r-1)/2$ many codes
\footnote{\alice\ do not require these codes to be valid.
Note that this does \emph{not} fill up the coding space at level $N$ since 
each code has capacity $2$. It only fill up half of the coding space at level $N$.
So this step cannot win the game \ref{optgame1} directly.}.
Or equivalently, to force
\begin{align}\nonumber
\text{Every $\sigma\in 2^N$ admitting at least 
$r/2$ many codes.}
\end{align}

\item First phase step 2:
cut off empty coding space.
This is the key step.
Since now every source $\sigma\in 2^N$ admits
at least $r/2$ many codes (and since $q/p\leq r/2$), we have 
at least $\frac{q}{2}2^N$ many codes
(in level $N$) are non empty.
Therefore,
\alice\ enumerates a small set $A_2 $
(promised by Claim \ref{optclaim10}
where we substitute set $C$ by the set of non empty codes in level $N$)
that will achieve the following.
For some set $\C_1\subseteq q\times 2^m$
(for some $m<N$) with  $|\C_1| \geq \frac{1}{2}q\cdot 2^m,$ 
some  $\h n$ sufficiently large depending on $m$
(with $m+\h n\leq N$),
the following holds:
\begin{align}\nonumber
\text{Every valid code in $   [\C_1]^{{\leq m+\h n}}  $
is non empty.}
\end{align}

\item First phase step 3:
fixing the coding function.
The key note here is that 
for every valid code $\rho\in \C_1$,
every source $\sigma\in p\times 2^m$
that is not (currently) coded by $\rho$,
\begin{align}\nonumber
&\text{The coding space $ [\rho]^{\leq m+\h n}$
has valid }\\ \nonumber
&\text{capacity \emph{one} for 
the source space $[\sigma]^\preceq$}
\end{align}
since every valid code
\footnote{For non valid code, it is possible to be empty. But this is useless since 
it could no longer validly code anything. } in this coding space is occupied
by at least one source outside $[\sigma]^\preceq$. 
Therefore, for a source $\sigma\in p\times 2^{[m, m+\h n]}$,
if \alice\ enumerate all current codes of $\sigma$,
then in the future the whole coding space $  [\C_1]^{\leq m+\h n}$ only have 
valid capacity one for source space $[\sigma]^\preceq$. But \alice\ has 
a static winning strategy for capacity one game, thus \alice\ could 
drive $\sigma$ away to $\C_0$ (by enumerating a small set). 
Indeed, \alice\ will achieve this at some middle level between $m,m+\h n$
(the reserved levels are to ensure that the cost of \alice\ for the capacity one game
etc has arbitrary small cost)

More specifically, let $\t n\in\omega$ be sufficiently large depending on $m$
(to be specified) and yet $\h n-3\t n$ is sufficiently large depending on 
$m+3\t n$. Let $f$ denote the current coding function of \baby.
\alice\ enumerate a small set $A_3$
to achieve:
\begin{align}\label{opteqfirstphasestep3target}
%&\text{For any coding function $\h f$ \baby\ presents (in the future),
%}\\ \nonumber
&\text{For any $\sigma \in p\times 2^{m+3\t n}$, in the future, $[\sigma]^\preceq$ cannot be solely
}\\ \nonumber
&\text{validly coded  in coding space }    [\C_1]^\preceq \setminus [f^{-1}(\sigma)]^\preceq.
\end{align}
This is done by applying the winning strategy for capacity one game
(Lemma \ref{optlemcapacityone}), namely
$(1,q2^{m+3\t n}, 1, \h n-3\t n)$-game \ref{optgame1}
(where roots of coding space are: $\rho\in (q\times 2^{m+3\t n})\setminus f^{-1}(\sigma)$)
with respect to   source space  $   [\sigma]^{\leq m+\h n}$ 
for each $\sigma\in p\times2^{m+3\t n}$
(i.e., The game is played $p2^{m+3\t n} $ times
with each time addressing one $\sigma\in p\times 2^{m+3\t n}$).

\item First phase step 4:
driving away each $\sigma\in p\times2^{m+\t n}$
(to outside $[\C_1]^\preceq$).
Since   (\ref{opteqfirstphasestep3target}) is achieved,
for each   $\t\sigma\in   [\sigma]^{=m+3\t n}$,
 enumerating all   codes of $\t\sigma$ entails 
that $\t\sigma$ can no longer be solely validly coded in $[\C_1]^\preceq$.
For each $\sigma\in p\times2^{m+\t n}$,
\alice\ choose a representative $ \tau_\sigma\in[\sigma]^{=m+ 3\t n} $ so that it does not admit too many codes in $[\C_1]^\preceq$
(namely $\leq 2q\cdot 2^{m+\t n}$).
Clearly, this accomplished 
the first phase target (\ref{optphaseonetarget}):
that each of $(\tau_\sigma: \sigma\in p\times 2^{m+\t n})$, therefore each $\sigma\in p\times2^{m+\t n}$,
 must be validly coded 
somewhere in $[\C_0]^\preceq$ in the future
(and \alice\ need not come back to $\C_1$ to maintain this).

\end{enumerate}
We look at the cost:
\begin{itemize}
\item
The cost in step 4 is bounded by:
\begin{align}\nonumber
\sum_{\sigma\in p\times 2^{m+\t n} }  \m(\tau_\sigma)\cdot |f^{-1}(\tau_\sigma)|&\leq p2^{m+\t n}\cdot \m(\tau_\sigma)\cdot 2q\cdot 2^{m+\t n}  \\ \nonumber
&\leq 2pq\cdot 2^{2m-\t n}
=o_{m,\t n\rightarrow\infty}(1).
\end{align}
\item The cost for step 3
is, by Lemma \ref{optlemcapacityone},
 $o_{m,\t n,\h n-3\t n\rightarrow\infty}(1)$.
\item The cost for step 2 is, by Claim \ref{optclaim10},
$o_{\h n, N\rightarrow\infty}(1)$
(note that it  depends on how large \alice\ want  $\h n$ to be).
\item The cost for step 1 is, by inductive hypothesis,
$o_{N,n\rightarrow\infty}(1)$ (which also depends on how large 
\alice\ want $N$ to be).
\end{itemize}
The choice of parameters are as following:
$n$ is sufficiently large in the sense that: there is an $N\leq n$ 
such that $n-N$ is sufficiently large depending on $N$ 
and $N$ is sufficiently large in the sense that:
there are $m<\t n<\h n $ with $m+\h n\leq N$ such that 
$\h n-3\t n$ is sufficiently large depending on $m+3\t n$
and $\t n$ is sufficiently large depending on $m$.
Thus, the overall cost for \alice\ in the first phase is 
$$o_{m,\t n\rightarrow\infty}(1)+
o_{m,\t n,\h n-3\t n\rightarrow\infty}(1)
+o_{\h n, N\rightarrow\infty}(1)+ o_{N,n\rightarrow\infty}(1)=o_{n\rightarrow\infty}(1).$$
It remains to address step 2 of the first phase.

 \ \\
 
\noindent\textbf{Cut off the empty codes.}

The following Claim says that
in step 2 of first phase, by enumerating 
a small set, there will be a (non small) part of the coding space and
 two levels so that every code in that part and between that 
 two levels is non empty.
 Recall the well known fact that given a 
 positive pruned tree $T\subseteq 2^{<\omega}$ (meaning $\m([T])>0$),
% This is alike the zero-one law for closed set
we can find a set $\h C\subseteq 2^m$ (for some $m$)
so that $\m([T]|[\h C])$ is arbitrary close to one
(in particular, $\m(T\cap 2^{m+\h n}|\h C)$ is close to one
for all $\h n\in\omega$).
The following Claim is a finite analog of this fact.

\begin{claim}
\label{optclaim9}
Let $\h n\in\omega$, $c>0$ and let $N\in\omega$ be sufficiently large
(depending on $ \h n, c$).
For every $\varepsilon>0$, every set $C\subseteq   2^N$ so that 
$\m(C)\geq \varepsilon $,
there exist  two 
 sets:
$\h C\subseteq 2^{\h m} $ (for some $\h m $ with 
$\h m+\h n<N$) with $\m(\h C)\geq \varepsilon $
and $A\subseteq 2^N$ with $\m(A)\leq c$
 such that the following holds.
For every $\rho\in  [\h C]^{=N}\setminus A$, 
we have $\rho\uhr m$ is a prefix of some element of $ C$
for all $\h m\leq m\leq \h m+\h n$.
\end{claim}
\begin{proof}
%This is alike the zero-one law,
%regarding $A$ as a closed set.
The proof simply proceeds by treating prefix of $C$ 
as the closed set in the above analog fact.
We inductively define a sequence of 
sets $( C_k\subseteq 2^{k\h n}:k\leq \t k)$ 
(with $\t k$ sufficiently large depending on $c$) so that each of them is a kind of complement of $C$.
Let $ C_1$ be the set of $\rho\in 2^{\h n}$
such that $\rho$ is not a prefix of any element of $C$.
i.e., $[\rho]^\preceq\cap C=\emptyset$.
Let $ C_2 $ be the set of $\rho\in 2^{2\h n}$ so that 
$\rho$ is not a prefix of any element of $C$ and 
$\rho\notin [ C_1]^\preceq$. And so forth.
Clearly $([ C_k]:k\leq \t k)$ are mutually disjoint open sets.
So
\begin{align}\nonumber
\sum_{k\leq \t k} \m( C_k)\leq 1-\m(C)\leq  1-\varepsilon.
\end{align}
Therefore,
since $\t k$ is sufficiently large depending on $c$,
 for some $k^*\leq \t k$, we have
$\m( C_{k^*})\leq c$.
Now let $$\h C:= 2^{(k^*-1)\h n}\setminus (\cup_{ k<k^*} [ C_{ k}]^\preceq)$$
we have $\m(\h C)\geq \varepsilon$.
Let $A:= [ C_{k^*}]^{=N}$, so $\m (A)=\m( C_{k^*})\leq c$.
Obviously, by choice of $ C_{k^*}$ and $\h C$,  
let $\h m:= (k^*-1)\h n$, we have for every $\rho\in    [\h C]^{=N}\setminus A$, every $\h m\leq m\leq \h m+\h n$,
$\rho\uhr m$ is a prefix of some element of $ C$.

\end{proof}

\begin{claim}\label{optclaim10}
Let $q\in\omega^{>0}$, $c>0$ and let $N\in\omega$ be sufficiently large
(depending on $q,c$).
For every $\varepsilon>0$, every set $C\subseteq  q\times 2^N$ so that 
$\m(C)\geq \varepsilon q $,
there exist  two 
 sets 
$\h C\subseteq q\times 2^{\h m} $ (for some $\h m<N$) with $\m(\h C)\geq \varepsilon q $
and $A\subseteq q\times 2^N$ with $\m(A)\leq c$
 such that the following holds.
For some $\h n$ sufficiently large depending on $\h m$ with $\h m+\h n\leq N$,
every $\rho\in  [\h C]^{=N}\setminus A$,
we have $\rho\uhr m$ is a prefix of some element of $ C$ for all $\h m\leq m\leq \h m+\h n$.
\end{claim}
\begin{proof}
Similarly as Claim \ref{optclaim9}
except that now we require $\h n $ to be sufficiently large depending on $\h m$.
Instead of letting
$  C_k $ to sit in level $k\h n$,
we choose a sequence $   m_1<\cdots <  m_{\t k}$ so that 
$  m_{k+1}-  m_k $ is sufficiently large depending on $  m_k$
and let $  C_k$ sit in level $ m_k$. 
 
\end{proof}

\section{Arbitrary capacity}
\label{optsecgeneralcapacity}

Fix $a\in\omega^{>0}$.
 Our goal in this section is a winning strategy for capacity $a$
(for arbitrary $p,q$) with arbitrary small cost:
 
\begin{lemma}\label{optlemcapacityamulticodetree}
For 
every $p,q\in\omega^{>0}$, every $c>0$, \alice\ has a winning strategy for the
$(p,q,a,n)$-game \ref{optgame1}  (for some $n$)
such that $\m(A)\leq c$.
\end{lemma}

 \noindent\textbf{Induction framework and new difficulty.}
 The proof proceeds by induction on $q/p$.
 For $q/p\leq 1/a$,
 the proof is obvious.
 There  are $\geq aq2^n$ many sources (in level $n$)
 and $q2^n$ many codes while each code can code at most $a$ sources.
 Thus, enumerating any code wins the game.
 Suppose   Lemma \ref{optlemcapacityamulticodetree}  holds for $q/p\leq r$
 and now $q/p\leq (1+1/a)r$ (where $1/a\leq r\in \mathbb R$).
 
 The new difficulty is that \alice\ do not have a static winning strategy for capacity 
 $a-1$ game.
 Therefore, she cannot, by enumerating a small set,
 entail that all certain source will always be validly coded over $\C_0$
in the future.
To this end, \alice\ shift between $\C_0$ and $\C_1$.
She works over $\C_1$ to drive all 
certain sources
(one representative $\tau_\sigma$ for each $\sigma\in p\times 2^{m+\t n}$) into $[\C_0]^\preceq$
(i.e., achieving an analog of (\ref{optphaseonetarget}) temporarily), say at round $t$;
then \alice\ applies the winning strategy for the smaller coding space $\C_0$.
But if later the working achievement 
(namely the analog of (\ref{optphaseonetarget})) fails,
\alice\ goes back to working over $\C_1$ until the working achievement is true again.

 \def\C{C}
 \ \\
 
\noindent \textbf{Wining strategy sketch.}
\begin{enumerate}
\item 
First phase and $\C_1$ part:
\alice\ will apply a strategy (with considerable effort and to be specified)
to achieve the following.
For some    $m\in\omega$ (to be specified)  
and sufficiently large  (depending on $m$ and to be specified) $\t n\in\omega$,
the following holds.
There is a partition of the coding space $q\times 2^m$ into two parts,
 namely $\C_0$ and $ \C_1$
 such that $|\C_1|\geq \frac{rp}{a}\cdot 2^m$.
Furthermore, \alice\ will maintain (by working over $\C_1$ and to be specified)
that (whenever \baby\ responded without violating the game rule over $\C_1$) there is always  a future round $t$
such that
\begin{align}\label{optphaseonetarget2}
&\text{No source $\sigma\in p\times 2^{m+\t n}$ 
is  solely }\\ \nonumber
&\text{validly coded in $[\C_1]^\preceq$ at   round $t$.
}
\end{align}
 The key is to show that 
 the cost for maintaining (\ref{optphaseonetarget2}) can be arbitrary small.

\item Second phase and $\C_0$ part:
Note that 
  $$|\C_0|\leq (q-\frac{rp}{a})\cdot 2^m\leq \bigg((1+\frac{1}{a})rp -\frac{rp}{a}\bigg)\cdot 2^m
 \leq rp\cdot 2^m.$$
At each round so that 
the working achievement (\ref{optphaseonetarget2}) holds,
the coding function \baby\ presents is seen as an
allowed action for $(p2^m, [rp2^m], a, \t n)$-game \ref{optgame1}.
Where the coding space consists of:  
the coding forest $$(  [\rho]^{\leq m+\t n }: \rho\in \C_0);$$
and the source space  consists of: 
the source forest $$( [\sigma]^{\leq m+\t n}: \sigma\in p\times 2^m).$$
Therefore, \alice\
apply the winning strategy 
for $(p2^m, [rp2^m], a, \t n)$-game \ref{optgame1}
at level $m+\t n$.
The induced game
complies to the definition of Game \ref{optgame1}.
By induction hypothesis,
\alice\ wins the game over $\C_0$ (with an arbitrary small cost
since $\t n$ is sufficiently large depending on $m$).

\end{enumerate}

Clearly, it remains to address the strategy of the first phase
which is the key of the proof.

 \ \\
 
\noindent\textbf{First phase strategy and working over $\C_1$.}

\begin{enumerate}
\item First phase step 1:
force more than $  r $ many codes for each source in level $N<n$. 
For each source $\sigma\in p\times 2^N$,  before $\sigma$ admits
more than $ r$ many codes,
apply the winning strategy (promised by inductive hypothesis)
with respect to source tree $[\sigma]^{=n}$
and (at most) $r$ many code trees (whose root are
codes of $\sigma$) to force $\sigma$ to admit
more than $ r$ many codes.

\item First phase step 2:
cut off empty coding space.
Since now every source $\sigma\in p\times 2^N$ admits
more than $r $ many codes, we have:
more than $\frac{rp}{a}2^N$ many codes
(in level $N$) are non empty.
Therefore,
\alice\ enumerates a small set $A_2 $
(promised by Claim \ref{optclaim10})
that will achieve the following.
For some set $\C_1\subseteq q\times 2^m$
(for some $m<N$) with  $|\C_1| \geq \frac{rp}{a}\cdot 2^m,$ 
some  $\h n$ sufficiently large depending on $m$
(with $m+\h n\leq N$),
the following holds:
\begin{align}\nonumber
\text{Every valid code $\rho\in   [\C_1]^{=m+\h n}  $
is non empty.}
\end{align}

\noindent The key note here is the following capacity minus one observation:
Let $f$ denote the current coding function of \baby.
For each $\sigma\in p\times 2^{[m,m+\h n]}$,
\begin{align}\label{opteqkey}
&\text{The coding space }    [\C_1]^{\leq m+\h n}\setminus [f^{-1}(\sigma) ]^{\leq m+\h n}\\ \nonumber
&\text{only have valid capacity $a-1$ for 
source space $  [\sigma]^{\leq m+\h n}$}
\end{align}
since
every valid code in this coding space is occupied
by (at least one) source outside $[\sigma]^\preceq$.

\item First phase step 3:
For each source $\sigma\in p\times 2^{m+\t n}$,
\alice\ choose a representative $ \tau_\sigma\in [\sigma]^{=m+ 3\t n}$ so that it does not admit too many codes in $[\C_1]^\preceq$
(namely $|f^{-1}(\tau_\sigma)|\leq aq\cdot 2^{m+\t n}$)
and enumerate $f^{-1}(\tau_\sigma)$ (for all $\sigma\in p\times 2^{m+\t n}$).
Combining with (\ref{opteqkey}),
we now have for each $\sigma\in p\times 2^{m+\t n}$
\footnote{Again, a non valid code in the coding space could have more than $a-1$
capacity for the source space over $\tau_\sigma$.
However, since every source must admit a valid code, we can simply ignore these non valid codes.}
\begin{align}\label{opteqkey2}
&\text{The coding space }   [\C_1]^{\leq m+\h n}    \text{  only have }\\ \nonumber
&\text{valid capacity $a-1$ for 
source space $  [\tau_\sigma]^{\leq m+\h n}$}
\end{align}

\item First phase step 4: maintaining
(\ref{optphaseonetarget2}).
 Note that
 at the round (\ref{optphaseonetarget2}) fails,
 say witnessed by $\sigma\in p\times2^{m+\t n}$
 (in particular, the source
 space $ [\tau_\sigma]^{\leq m+\h n}$
 is solely validly coded over $\C_1$), the coding function
 presented by \baby\ can be seen as an allowed move for 
  $(1,q2^{m+3\t n}, a-1, \h n-3\t n)$-game \ref{optgame1}
where the source space is $ [\tau_\sigma]^{\leq m+\h n}$
and the coding space is  $[\C_1]^{[m+3\t n, m+\h n]}  $.
 Therefore, \alice\ resumes the winning strategy 
 for the game to re-achieve (\ref{optphaseonetarget2}).
Note that \alice\ is thereby playing $p2^{m+\t n}$ many games
(one for each $\tau_\sigma$)
simultaneously.

\end{enumerate}
We look at the cost:
\begin{itemize}
\item The cost in step 4 is bounded by:
\begin{align}
&p2^{m+\t n}\cdot\text{cost of each game given by some $(\tau_\sigma:\sigma\in p\times2^{m+\t n})$}\\ \nonumber
=& p2^{m+\t n}\cdot o_{m,\t n,\h n-3\t n\rightarrow\infty}(1)
 \\ \nonumber
 =& o_{m,\t n,\h n-3\t n\rightarrow\infty}(1).
\end{align}
\item
The cost in step 3 is bounded by:
\begin{align}\nonumber
\sum_{\sigma\in p\times 2^{m+\t n} }  \m(\tau_\sigma)\cdot |f^{-1}(\tau_\sigma)|&\leq p2^{m+\t n}\cdot \m(\tau_\sigma)\cdot aq\cdot 2^{m+\t n}  \\ \nonumber
&\leq apq\cdot 2^{2m-\t n}
=o_{m,\t n\rightarrow\infty}(1).
\end{align}
 
\item The cost for step 2 is, by Claim \ref{optclaim10},
$o_{\h n, N\rightarrow\infty}(1)$
(which also depends on how large \alice\ want  $\h n$ to be).
\item The cost for step 1 is, by inductive hypothesis,
$o_{N,n\rightarrow\infty}(1)$ (which also depends on how large 
\alice\ want $N$ to be).
\end{itemize}
The choice of parameters are, again, as following:
$n$ is sufficiently large in the sense that: there is an $N\leq n$ 
such that $n-N$ is sufficiently large depending on $N$ 
and $N$ is sufficiently large in the sense that:
there are $m<\t n<\h n $ with $m+\h n\leq N$ such that 
$\h n-3\t n$ is sufficiently large depending on $m+3\t n$
and $\t n$ is sufficiently large depending on $m$.
Thus, the overall cost for \alice\ in the first phase and over $\C_1$ is 
$$o_{m,\t n\rightarrow\infty}(1)+
o_{m,\t n,\h n-3\t n\rightarrow\infty}(1)
+o_{\h n, N\rightarrow\infty}(1)+ o_{N,n\rightarrow\infty}(1)=o_{n\rightarrow\infty}(1).$$
 
We have closed the induction, thus
concludes the
proof of Theorem \ref{optpropwin}.

\section*{Acknowledgement}
\noindent\textbf{Funding.} Lu Liu is supported by  National Natural Science Foundation of China grant 12471002.

\noindent\textbf{Use of AI.} This research is aided by AI. 
Authors transform the original question (Corollary \ref{optth0})  to 
a finite game question (Theorem \ref{optpropwin}) and then proved a special case 
for Theorem \ref{optpropwin}
 namely capacity $a=2$ and $p\leq q$.
Author then feed this proof to AI who then provided a 
somewhat different  proof  of Theorem \ref{optpropwin}
for the general case, which is   presented in this paper. 
Author read and checked all results and proofs provided by AI and take full responsibility.
 We also used AI for checking proofs,  literature search
 and polishing.

\end{document}